\documentclass[12pt,reqno]{amsart}

\usepackage[ansinew]{inputenc}
\usepackage{graphicx}
\usepackage[bookmarksnumbered,colorlinks, linkcolor=blue, citecolor=red, pagebackref, bookmarks, breaklinks]{hyperref}

\newtheorem{thm}{Theorem}[section]
\newtheorem{cor}[thm]{Corollary}
\newtheorem{lem}[thm]{Lemma}
\newtheorem{prop}[thm]{Proposition}

\theoremstyle{definition}

\newtheorem{ex}[thm]{Example}

\theoremstyle{remark}
\newtheorem{rem}[thm]{Remark}

\numberwithin{equation}{section}

\def\diver{\mathop{\text{\normalfont div}}}

\newcommand{\R}{\mathbb{R}}
\newcommand{\N}{\mathbb{N}}

\newcommand{\ve}{\varepsilon}

\newcommand{\LL}{\mathcal{L}}

\begin{document}

\title[Orlicz eigenvalues]{On the first eigenvalue of the generalized laplacian}

\author[J. Fern\'andez Bonder and A. Salort]{Juli\'an Fern\'andez Bonder and Ariel M. Salort}

\address{Departamento  de Matem\'atica, FCEN -- Universidad de Buenos Aires, and\hfill\break\indent Instituto de C\'alculo -- CONICET\hfill\break\indent $0+\infty$ building, Ciudad Universitaria (1428), Buenos Aires, Argentina.}

\email{{\tt jfbonder@dm.uba.ar, asalort@dm.uba.ar}\hfill\break\indent {\it Web page:} {\tt http://mate.dm.uba.ar/$\sim$jfbonder, http://mate.dm.uba.ar/$\sim$asalort}}

\begin{abstract}
In this work we investigate the energy of minimizers of Rayleigh-type quotients of the form
$$
\frac{\int_\Omega A(|\nabla u|)\, dx}{\int_\Omega A(|u|)\, dx}.
$$

These minimizers are eigenfunctions of the generalized laplacian defined as $\Delta_a u = \diver\left(a(|\nabla u|)\frac{\nabla u}{|\nabla u|}\right)$ where $a(t)=A'(t)$ and the Rayleigh quotient is comparable to the associated eigenvalue. On the function $A$ we only assume that it is a Young function but no $\Delta_2$ condition is imposed. 

Since the problem is not homogeneous, the energy of minimizers is known to strongly depend on the normalization parameter $\alpha =\int_\Omega A(|u|)\, dx$. In this work we precisely analyze this dependence and show differentiability of the energy with respect to $\alpha$ and, moreover, the limits as $\alpha\to 0$ and $\alpha\to \infty$ of the Rayleigh quotient.

The nonlocal version of this problem is also analyzed.
\end{abstract}

\maketitle
\tableofcontents

\section{Introduction}

Let $A\colon \R_+ \to \R_+$ be convex function superlinear at 0 and at $\infty$.  Given an open and bounded set $\Omega\subset \R^n$ and a fixed normalization parameter $\alpha>0$, we consider the {\em energy function} $E(\alpha)$ defined as
\begin{equation} \label{minimi}
E(\alpha) = \inf\left\{\int_\Omega A(|\nabla u|)\, dx \colon u\in C_c^\infty(\Omega), \int_\Omega A(|u|)\, dx = \alpha\right\}.
\end{equation}
Without any further growth assumption on $A$, in \cite{GLMS99,GM02} it is proved that the minimization problem $E(\alpha)$ is solvable, i.e., there exists a function $u_\alpha$ for which the infimum in $E(\alpha)$ is achieved. Furthermore, through a generalized version of the Lagrange's Multiplier Theorem, in \cite{GM02} (see also \cite{T00}) it is proved that $u_\alpha$ satisfies in the weak sense the equation
\begin{align}\label{autov}
\begin{cases}
-\Delta_a u = \lambda(\alpha) \frac{a(|u|)}{|u|} u &\quad \text{ in } \Omega\\
u=0& \quad \text{ in } \partial\Omega,
\end{cases}
\end{align}
where the eigenvalue $\lambda(\alpha)$ is the Lagrange multiplier, $a(t)=A'(t)$ and $\Delta_a u:=\diver\left(\frac{a(|\nabla u|)}{|\nabla u|}\nabla u\right)$ is the so-called \emph{generalized Laplacian} or the \emph{$a-$Laplacian operator}. 

Existence of solutions to \eqref{minimi} and eigenpairs to \eqref{autov} assuming diverse growth behaviors on $A$  were widely studied, see for instance \cite{GLMS99, GM02, M12, MT99,  S22, T00}. These problems, typically are non homogeneous and in general, the eigenvalue $\lambda(\alpha)$ does not coincide with the ratio $E(\alpha)/\alpha$. One known exception arises when the problem is homogeneous, that is, when $A(t)=t^p$ for some $p>1$ and hence  \eqref{autov} turns out to be  the well-known eigenvalue problem for the Dirichlet $p-$Laplace operator.  In this case the first eigenvalue of such a problem admits the variational characterization
$$
\lambda(1)=\lambda(\alpha) =\frac{E(\alpha)}{\alpha}=\inf_{0\neq u\in C_c^\infty(\Omega)} \frac{\int_\Omega |\nabla u|^p dx}{ \int_\Omega |u|^p\, dx}
$$
and it is independent of the normalization parameter $\alpha>0$. For further information, see for instance \cite{L90, L06, L08}.

Although, in general, eigenvalues and energy may not coincide, when the Young function $A$ satisfies the $\Delta_2-$condition (see Section \ref{sec.prel} for the definition), it can be readily proved that the following comparison estimate  
\begin{equation}\label{autov.trivial}
c \lambda(\alpha)\le \frac{E(\alpha)}{\alpha}\le C \lambda(\alpha),
\end{equation}
holds for some  constants $c, C>0$ depending only on $A$. Moreover, in Section \ref{sec.elementary} the following elementary  inequality  is proved to hold:
\begin{equation}\label{cota.trivial}
\min\left\{\alpha^{p-1}, \alpha^{\frac{1}{p}-1}\right\} E(1)\le \frac{E(\alpha)}{\alpha}\le \max\left\{\alpha^{p-1}, \alpha^{\frac{1}{p}-1}\right\} E(1),
\end{equation}
where $p>1$ is a suitable constant depending only on $A$.  However, while this estimate is straightforward, it lacks the depth necessary to provide meaningful insights into the behavior of $E(\alpha)$ as $\alpha\downarrow 0$ and $\alpha\uparrow\infty$. This lack of information serves as the primary motivation for this paper.

Our first results concern the regularity of the curve $E(\alpha)$. In Theorem  \ref{E.Lip}, without any additional assumptions regarding the behavior of $A$, we prove that the energy $E(\alpha)$ is a Lipschitz continuous function of the parameter $\alpha>0$. Moreover, assuming a technical but fairly natural condition on $A$ we are able to obtain its differentiability. More precisely, assuming that the map $t\mapsto ta(t)$ is convex, in Theorem \ref{teo.dif} we prove that $E(\alpha)$ is differentiable for any $\alpha>0$ and the following formulas hold:
\begin{align*}
\frac{dE}{d\alpha}(\alpha)=\lambda(\alpha), \qquad \frac{d}{d\alpha}\left(\frac{E(\alpha)}{\alpha}\right)= \frac{1}{\alpha} \left(\lambda(\alpha)-\frac{E(\alpha)}{\alpha}\right),
\end{align*}
In these expressions, the eigenvalue $\lambda(\alpha)$ and the energy quotient $E(\alpha)/\alpha$ coincide if and only if both quantities are constants independent of $\alpha$. This is precisely the scenario where $A$ is homogeneous.

In Section \ref{sec.asymptotic} we analyze the behavior of the ratio $E(\alpha)/\alpha$ as $\alpha\downarrow 0$ and $\alpha\uparrow \infty$. Our main result establishes that the behavior of that quotient  can be described in terms of the the Matuszewska-Orlicz functions associated to the Young function $A$, which are defined as follows
$$
M_0(t) = \limsup_{\tau\to 0} \frac{A(\tau t)}{A(\tau)}\qquad \text{and} \qquad M_\infty(t) = \limsup_{\tau\to\infty} \frac{A(\tau t)}{A(\tau)},
$$
(see Section \ref{sec.prel} for further information and properties of these function). We highlight that the behavior of the energy quotient only takes into account the local behavior of the Young function $A$ near $0$ and near $\infty$. More precisely, in Theorem \ref{teo1} we prove that if  $A$ satisfies the $\Delta_2$ condition near $i=0,\infty$, then 
$$
\limsup_{\alpha\to i} \frac{E(\alpha)}{\alpha} = \inf_{\tau>0} \frac{E_{i}(\tau)}{\tau} 
$$ 
where
$$
E_{i}(\tau) = \inf\left\{\int_\Omega M_i(|\nabla u|)\, dx\colon u\in C_c^\infty(\Omega), \, \int_\Omega M_i(|u|)\, dx = \tau\right\}.
$$
This result can be refined. We prove that under a fairly general additional assumption on $A$, there exist numbers $p_i\in [1,\infty)$ for which  $M_i(t)=t^{p_i}$, and these result in
 $$
 \lim_{\alpha\to i} \frac{E(\alpha)}{\alpha} = \lambda_{p_i},
 $$
where $\lambda_p$ denotes the first eigenvalue of the $p-$Laplacian with Dirichlet boundary conditions in $\Omega$.
 
Although in the previous result the assumption of the $\Delta_2$ condition on $A$ seems restrictive, in fact this is somehow the interesting case of study since when this condition is removed, in broad terms, the Matuszewska-Orlicz functions associated to the Young function $A$ becomes trivial:
$$
M_i(t)=\begin{cases}
0 & \text{if } t<1\\
1 & \text{if } t=1\\
\infty & \text{if } t>1\\
\end{cases}
$$
for $i=0,\infty$. In this case, under suitable assumptions on the domain, the limit of $E(\alpha)/\alpha$ can be characterized. More precisely, if $\Omega\subset\R^n$ is  such that the inner radius is larger than 1, in Theorem \ref{lim.sin.delta2}  we obtain that 
$$
\lim_{\alpha\to i} \frac{E(\alpha)}{\alpha} = 0.
$$

\medskip

Applying the same technique, we can extend our approach to handle the nonlocal version of \eqref{minimi} within the framework of fractional Orlicz-Sobolev spaces of order $s\in (0,1)$ introduced in \cite{FBS19}. Specifically, for a fixed  $\alpha>0$ we analyze the minimization problem defined by
\begin{equation} \label{minimi.nl.intro}
E^s(\alpha) = \inf\left\{\iint_{\R^{2n}} A\left( \frac{|u(x)-u(y)|}{|x-y|^s}\right)\, d\nu_n \colon u\in C^\infty_c(\Omega), \int_\Omega A(|u|)\, dx = \alpha\right\}
\end{equation}
where we denote  $d\nu_n :=|x-y|^{-n}dxdy$. As  in the local case, without any growth assumption on the Young function $A$, in \cite{SV22}  it is proved the existence of a function $u^s_\alpha$ for which the infimum in \eqref{minimi.nl.intro} is attained. Moreover,  this function satisfies in the weak sense the equation
\begin{align}\label{autov.nl.intro}
\begin{cases}
(-\Delta_a)^s u= \lambda^s(\alpha) \frac{a(|u|)}{|u|} u &\quad \text{ in } \Omega\\
u=0& \quad \text{ in } \partial\Omega,
\end{cases}
\end{align}
where the fractional $a-$Laplacian of order $s\in (0,1)$ is defined as
$$(-\Delta_a)^s u:=\text{p.v.}\, \displaystyle\int_{\R^n} a\left( \frac{|u(x)-u(y)|}{|x-y|^s} \right)\frac{u(x)-u(y)}{|u(x)-u(y)|} \frac{dy}{|x-y|^{n+s}}
$$
and $\lambda^s(\alpha)$ is the Lagrange's multiplier. Problems \eqref{minimi.nl.intro} and \eqref{autov.nl.intro} (and further variants) were widely studied, see for instance \cite{BOS23, BS21, FBSp24, FBS19, FBSV23, S20, SV22}.

\subsection*{Organization of the paper}
After this introduction, in Section \ref{sec.prel} we review the results on Young functions and Orlicz and Orlicz-Sobolev spaces needed in the rest of the article. Most results in this section are well known and in the cases where we were not able to find a precise reference, a full proof is provided.
Then, we include a section (Section \ref{sec.elementary}) where, assuming the $\Delta_2-$condition on $A$, we deduce some elementary estimates on the energy $E(\alpha)$ and on the eigenvalue $\lambda(\alpha)$. In particular, we obtain the estimates \eqref{autov.trivial} and \eqref{cota.trivial}.

Section \ref{sec.regularity} is devoted to the study of the regularity properties of the energy function $E(\alpha)$. The main result being Theorem \ref{teo.dif}.

The behavior of the energy quotient $E(\alpha)/\alpha$ as $\alpha\to 0$ and $\alpha\to \infty$ is analyzed in Sections \ref{sec.asymptotic} and \ref{sec.not.delta2}. In Section \ref{sec.asymptotic} we analyze the more interesting case where $A$ verifies the $\Delta_2-$condition either at 0 or at $\infty$ and in Section \ref{sec.not.delta2} the case in which $A$ does not verifies the $\Delta_2-$condition.

Finally, in Section \ref{sec.nonlocal} we state our results for the nonlocal case.

\section{Preliminaries} \label{sec.prel}

In this section we collect all of the preliminaries needed in this work. Almost all of the results are well known and we try to provide references in all the cases. There are a couple of new, or not so well known results and in those cases full proofs are provided.

The standard reference for this section is the book \cite{KR}.

\subsection{Young functions}
A function $A\colon \R_+\to\R_+$ is called a \emph{Young function} if it can be written as
$$
A(t)=\int_0^t a(\tau)\, d\tau,
$$
where $a\colon \R_+\to\R_+$ is nondecreasing, right continuous, $a(t)>0$ if $t>0$, $a(0)=0$ and $a(t)\to\infty$ as $t\to\infty$. 

From this definition, it can readily be checked that $A$ is convex, $A(0)=0$ and it is superlinear at $0$ and at $\infty$, i.e.
$$
\lim_{t\to 0} \frac{A(t)}{t}=0,\quad \lim_{t\to\infty} \frac{A(t)}{t}=\infty.
$$

Observe that the convexity of the Young function immediately gives that
\begin{equation} \label{convexidad}
A(\tau t) \leq \tau A(t) \quad \text{ for } 0<\tau<1, \qquad 
A(\tau t) \geq \tau A(t) \quad \text{ for } \tau>1.
\end{equation}

Given a Young function $A$, an important associated function is the so-called \emph{complementary function}, denoted by $\bar A$, that is defined as
$$
\bar A(t):=\sup \{\tau t- A(\tau)\colon \tau\geq 0\}.
$$
The complementary function $\bar A$ is also a Young function.

Observe that $\bar A$ is the optimal Young function in the {\em Young-type inequality}
$$
\tau t\le \bar A(\tau) + A(t)
$$
for all $\tau,t\geq 0$. It is also easy to see that $\bar{\bar A} = A$.

We now define the class 
$$
\Delta_2^\infty :=\left\{A\colon \R_+\to\R_+\colon \text{Young function such that }\limsup_{t\to\infty} \frac{A(2t)}{A(t)} < \infty\right\}
$$
and
$$
\Delta_2^0 :=\left\{A\colon \R_+\to\R_+\colon \text{Young function such that }\limsup_{t\to 0} \frac{A(2t)}{A(t)} < \infty\right\}.
$$

It is easy to see that $A\in \Delta_2^\infty$ if and only if there exists constants $C_\infty\ge 2$ and $T_\infty>0$ such that
$$
A(2t)\leq C_\infty A(t) \quad \text{ for } t\geq T_\infty,
$$
and $A\in \Delta_2^0$ if and only if there exists $C_0\ge 2$ and $T_0>0$ such that
$$
A(2t)\leq C_0 A(t) \quad \text{ for } t\leq T_0.
$$
Finally, we define the class 
$$
\Delta_2 = \Delta_2^0\cap \Delta_2^\infty
$$
and $A\in \Delta_2$ if and only if
$$
\sup_{t>0} \frac{A(2t)}{A(t)}<\infty.
$$

In \cite[Theorem 4.1]{KR} it is shown that $A\in \Delta_2^\infty$ (or $A\in \Delta_2^0$, or $A\in \Delta_2$) if and only if there exists $p>1$ such that
\begin{equation}\label{def.p}
\frac{t a(t)}{A(t)}\le p\quad \text{for all } t\ge T_\infty \text{ (or $t\le T_0$, or $t>0$)}.
\end{equation}

In order to analyze the behavior of a Young function $A$ at 0 and at $\infty$, we recall the Matuszewska-Orlicz functions that are defined as
\begin{equation} \label{matu}
M_0(t) = \limsup_{\tau\to 0} \frac{A(\tau t)}{A(\tau)}\qquad \text{and} \qquad M_\infty(t) = \limsup_{\tau\to\infty} \frac{A(\tau t)}{A(\tau)},
\end{equation}
These functions were introduced by Matuszewska and Orlicz in \cite{MO60} and the reader can consult with the book \cite{Maligranda} for the results presented here. For $i=0, \infty$, we have that $M_i$ is convex and sub-multiplicative, that is $M_i\colon \R_+\to [0,\infty]$ verify
$$
M_i(0)=0,\quad M_i(1)=1,\quad M_i(\tau t)\le M_i(\tau)M_i(t).
$$
When $A\in \Delta_2^i$, it follows that $M_i(t)<\infty$ for every $t>0$.


In the case where $A\not\in \Delta_2^i$ for some $i=0,\infty$, the Matuszewska function $M_i$ turns out to be trivial. We were not able to find this result in the literature, so we include a proof of this fact. 
\begin{prop}\label{Minf.trivial}
Let $A$ be a Young function. Then, if $A\not\in \Delta_2^i$, $i=0,\infty$, we have that
$$
M_i(t)=\begin{cases}
1 & \text{if } t=1\\
\infty & \text{if } t>1.
\end{cases}
$$
Moreover, if $A$ satisfies the slightly stronger assumption 
\begin{equation}\label{s.not.delta2}
\liminf_{t\to i} \frac{A(2t)}{A(t)}=\infty,
\end{equation}
then $M_i(t) = 0$  if $0<t<1$.
\end{prop}

\begin{proof}
We prove the result for $i=\infty$, the other case being analogous.

For $t>1$ we argue as in \cite[Theorem 4.1]{KR}, to obtain
$$
A(st)\ge \int_s^{st} a(\tau)\, d\tau\ge (t-1)sa(s).
$$
Then,
\begin{equation}\label{t>1}
\limsup_{s\to\infty} \frac{sa(s)}{A(s)} \le \frac{1}{t-1} \limsup_{s\to\infty}\frac{A(st)}{A(s)} = \frac{1}{t-1} M_\infty(t).
\end{equation}
Therefore $A\not\in \Delta_2^\infty$ implies that $M_\infty(t)=\infty$ for $t>1$.

Observe that $M_\infty(1)=1$ by definition.

Now, if \eqref{s.not.delta2} holds, then the limit in \eqref{t>1} exists. And so, for $0<t<1$, we have
$$
\frac{A(st)}{A(s)} = \frac{1}{\frac{A(st \frac{1}{t})}{A(st)}}.
$$
Therefore the result follows by taking the limit in the former expression.
\end{proof}

As we mentioned, if $A\in \Delta_2^i$, then $M_i(t)<\infty$ for every $t>0$. In the case where the limit in the definition of $M_i$ \eqref{matu} exists, these Matuszewska function are power functions. This is the content of the next result.
\begin{prop}\label{M.potencia}
Assume $A\in \Delta_2^i$ for some $i=0,\infty$. Assume moreover that the limit in \eqref{matu} exists. Then, there exists $p_i\in [1,\infty)$ such that
$$
M_i(t)=t^{p_i}.
$$
\end{prop}

\begin{proof}
First observe that, by \eqref{convexidad}, it follows that
\begin{equation}\label{cota.M.trivial}
M_i(t)\le t \text{ for } 0<t<1\quad \text{and}\quad M_i(t)\ge t \text{ for } t>1.
\end{equation}

Now, the proof follows just observing that if the limit in \eqref{matu} exists, then $M_i(t)$ is multiplicative, i.e. $M_i(st)=M_i(s)M_i(t)$ and $M_i(1)=1$. 

Then, if we define $v_i\colon \R\to\R$ by $v_i(t)=\ln(M_i(e^t))$, the function $v_i$ is additive, i.e. $v_i(s+t)=v_i(s)+v_i(t)$ for any $s, t\in\R$. It is a well known fact that measurable additive functions are linear, therefore there exists $p_i\in \R$ such that $v_i(t)=p_i t$ from where we can deduce that $M_i(t) = t^{p_i}$. This fact, together with \eqref{cota.M.trivial}, implies that $p_i\ge 1$.
\end{proof}

\subsubsection{Examples}
\begin{ex}
Let $p>1$, and assume that
$$
A(t)\sim t^p \quad\text{when } t\ll 1.
$$
Then $A\in \Delta_2^0$ and $M_0(t)=t^p$.

If we assume that
$$
A(t)\sim t^p \quad\text{when } t\gg 1,
$$
then $A\in \Delta_2^\infty$ and $M_\infty(t)=t^p$.

As a special case we have the $(p,q)-$Laplacian: if $1<p<q<\infty$
$$
A(t)=\frac{t^p}{p}+ \frac{t^q}{q},\quad \text{then } A\in \Delta_2 \quad \text{and } M_0(t)=t^p, \quad    M_\infty(t)=t^q.
$$
\end{ex}

\begin{ex} Let $\alpha\ge 0$ and $p\ge 1$. Then if
$$
A(t)\sim t^p \ln^\alpha t \quad \text{when } t\ll 1,
$$
then $A\in \Delta_2^0$ and $M_0(t)=t^p$. If
$$
A(t)\sim t^p \ln^\alpha t \quad \text{when } t\gg 1,
$$
then $A\in \Delta_2^\infty$ and $M_\infty(t)=t^p$.

As a special case, if
$$
A(t)=\frac{t^p}{p}\ln^\alpha(1+t^r),
$$
for some $r>0$, then $A\in \Delta_2$ and $M_0(t)=t^{p+r\alpha}$, $M_\infty(t)=t^{p}$.
\end{ex}

\begin{ex} For $\alpha, \beta\ge 0$ and $p\ge1$, assume that
$$
A(t)\sim t^p \ln^\alpha(t) \ln^\beta(\ln t)\quad \text{when } t\gg 1.
$$
Then $A\in \Delta_2^\infty$ and $M_\infty(t)=t^p$.
\end{ex}

\begin{ex} For $n\in\N$, define $A(t)=e^t - \sum_{k=0}^{n-1} \frac{t^k}{k!}$. Then $A\in \Delta_2^0$, but $A\not\in \Delta_2^\infty$. Moreover,
$$
M_0(t)=t^n \quad\text{and}\quad M_\infty(t)=\begin{cases}
0 & \text{if } 0<t<1\\
1 & \text{if } t=1\\
\infty & \text{if } t>1.
\end{cases}
$$
\end{ex}

\begin{ex} For $\alpha>0$, assume that
$$
A(t)\sim e^{-t^{-\alpha}}\quad \text{when } t\ll 1.
$$
Then $A\not\in \Delta_2^0$ and $M_0(t)=\begin{cases}
0 & \text{if } 0<t<1\\
1 & \text{if } t=1\\
\infty & \text{if } t>1.
\end{cases}$
\end{ex}

\begin{ex} Assume that
$$
A(t)\sim e^{e^t} \quad \text{when } t\gg 1.
$$
Then, $A\not\in \Delta_2^\infty$ and $M_\infty(t)=\begin{cases}
0 & \text{if } 0<t<1\\
1 & \text{if } t=1\\
\infty & \text{if } t>1.
\end{cases}$
\end{ex}

Observe that all of the examples in this subsection satisfy the hypotheses of Proposition \ref{Minf.trivial} or Proposition \ref{M.potencia}.

\subsection{Orlicz and Orlicz-Sobolev spaces}
As we mentioned in the introduction on this section, the main reference for Orlicz spaces is the book \cite{KR}. As for Orlicz-Sobolev spaces, the reader can consult, for instance, with \cite{Gossez}. 

Fractional order Orlicz-Sobolev spaces, as we will use them here, were introduced in \cite{FBS19} and then further analyze by several authors. The results used in this paper are found in \cite{ACPS2, ACPS} and in \cite{FBSp24}.

\subsubsection{Orlicz spaces}
Given $\Omega\subset \R^n$ a bounded domain and a Young function $A$, the Orlicz class is defined as
$$
\LL^A(\Omega) :=\left\{u\in L^1_\text{loc}(\Omega)\colon \int_\Omega A(|u|)\, dx<\infty\right\}.
$$ 
Then the Orlicz space $L^A(\Omega)$ is defined as the linear hull of $\LL^A(\Omega)$ and is characterized as
$$
L^A(\Omega) = \left\{u\in L^1_\text{loc}(\Omega)\colon \text{ there exists } k>0 \text{ such that }  \int_\Omega A\left(\frac{|u|}{k}\right)\, dx<\infty\right\}.
$$
In general the Orlicz class is strictly smaller than the Orlicz space, and $\LL^A(\Omega) = L^A(\Omega)$ if and only if $A\in \Delta_2^\infty$.

The space $L^A(\Omega)$ is a Banach space when it is endowed, for instance, with the {\em Luxemburg norm}, i.e.
$$
\|u\|_{L^A(\Omega)} = \|u\|_A :=\inf\left\{k>0\colon  \int_\Omega A\left(\frac{|u|}{k}\right)\, dx\le 1\right\}.
$$

This space $L^A(\Omega)$ turns out to be separable if and only if $A\in \Delta_2^\infty$.

An important subspace of $L^A(\Omega)$ is $E^A(\Omega)$ that it is defined as the closure of the functions in $L^A(\Omega)$ that are bounded. This space is characterized as
$$
E^A(\Omega) = \left\{u\in L^1_\text{loc}(\Omega)\colon  \int_\Omega A\left(\frac{|u|}{k}\right)\, dx<\infty \text{ for all } k>0\right\}.
$$

This subspace $E^A(\Omega)$ is separable, and we have the inclusions
$$
E^A(\Omega)\subset \LL^A(\Omega)\subset L^A(\Omega)
$$
with equalities if and only if $A\in \Delta_2^\infty$.

The following duality relation holds
$$
(E^A(\Omega))^* = L^{\bar A}(\Omega),
$$
where the equality is understood  via the standard duality pairing. Observe that this automatically implies that $L^A(\Omega)$ is reflexive if and only if $A, \bar A\in \Delta_2^\infty$.

In the sequel, we will need a well known lemma (see for instance \cite{KR}), but we include the details for completeness.
\begin{lem}\label{cota.prelim}
Let $A$ be a Young function and $u\in L^A(\Omega)$. Then
$$
\|u\|_A\le \max\left\{1; \int_\Omega A(|u|)\, dx\right\}.
$$
\end{lem}

\begin{proof}
We can assume that $u\in \LL^A(\Omega)$, otherwise there is nothing to prove. Also, we can assume that $\|u\|_A >1$.
Then,
$$
\int_\Omega A(|u|)\,dx = \int_\Omega A\left(\frac{|u|}{\|u\|_A} \|u\|_A\right)\,dx \geq \|u\|_A,
$$
where we have used \eqref{convexidad} and the definition of the Luxemburg norm. 
\end{proof}

\subsubsection{Orlicz-Sobolev spaces}
Once the Orlicz spaces are defined, one can easily define the Orlicz-Sobolev spaces in the usual way, namely
$$
W^1L^A(\Omega) = \{u\in L^A(\Omega)\colon \partial_{x_i}u\in L^A(\Omega), i=1,\dots,n\}
$$
and
$$
W^1E^A(\Omega) = \{u\in E^A(\Omega)\colon \partial_{x_i}u\in E^A(\Omega), i=1,\dots,n\}.
$$
These two spaces coincide only if $A\in \Delta_2^\infty$ and in this case we will denote the Orlicz-Sobolev space as
$$
W^{1, A}(\Omega) = W^1L^A(\Omega) = W^1E^A(\Omega).
$$

These spaces are endowed with the natural norm
$$
\|u\|_{W^1L^A(\Omega)} = \|u\|_{1, A} = \|u\|_A + \|\nabla u\|_A.
$$
With this norm, $W^1L^A(\Omega)$ is a Banach space and $W^1E^A(\Omega)$ is a closed subspace.

In order to work with Dirichlet boundary conditions it is necessary to define the Orlicz-Sobolev functions that vanish at the boundary. So, we define $W^1_0L^A(\Omega)$ as the closure of $C^\infty_c(\Omega)$ with respect to the topology $\sigma(W^1L^A(\Omega), W^1E^{\bar A}(\Omega))$ and $W^1_0 E^A(\Omega)$ as the closure of $C^\infty_c(\Omega)$ in norm topology. 

Again, when $A\in \Delta_2^\infty$ these spaces coincide and will be denoted as
$$
W^{1, A}_0(\Omega) = W^1_0 L^A(\Omega) = W^1_0 E^A(\Omega).
$$
These spaces are reflexive if and only if both $A, \bar A\in \Delta_2^\infty$. In any other case, what we have is that there exists a Banach space $X$ such that $X^* = W^1_0 L^A(\Omega)$. This Banach space $X$ can be easily characterized in terms of distributions, but it will not be used in this article.

In \cite{GLMS99} the following is proved.
\begin{prop} \label{embeddings}
Let $A$ be a Young function. Then the embedding $W^1_0 L^A(\Omega) \subset L^A(\Omega)$ is compact.
\end{prop}

\subsubsection{Fractional Orlicz-Sobolev spaces}
The construction of the fractional order Orlicz-Sobolev spaces is almost the same as the Orliz-Sobolev spaces. We will make only a sketch and the reader can consult with \cite{FBSp24}.

Given a fractional parameter $s\in (0,1)$, we define the H\"older quotient of a function $u\in L^A (\Omega)$ as
$$
D^su(x,y) = \frac{u(x)-u(y)}{|x-y|^s}.
$$
Then, the fractional Orlicz-Sobolev space of order $s$ is defined as
$$
W^sL^A(\R^n) := \{u\in L^A(\R^n)\colon D^su\in L^A(\R^{2n}, d\nu_n)\},
$$
where $d\nu_n = |x-y|^{-n}dxdy$ and
$$
W^sE^A(\R^n) := \{u\in E^A(\R^n)\colon D^su\in E^A(\R^{2n}, d\nu_n)\}.
$$
Since we are now working in the whole space, in order for these spaces to coincide we require that $A\in \Delta_2$. In this case, we denote
$$
W^{s, A}(\R^n) =W^sL^A(\R^n)= W^sE^A(\R^n).
$$
As in the previous case, we have that $W^s L^A(\R^n)$ is reflexive if and only if $A, \bar A\in \Delta_2$.

In these spaces the norm considered is
$$
\|u\|_{W^sL^A(\R^n)} = \|u\|_{s, A} = \|u\|_A + \|D^su\|_{A, d\nu_n}.
$$
Again, with this norm, $W^sL^A(\R^n)$ is a Banach space and $W^sE^A(\R^n)$ is a closed subspace.

The space $W^s_0 L^A(\Omega)$ is then defined as the closure of $C^\infty_c(\Omega)$ with respect to the topology $\sigma(W^sL^A(\R^n), W^sE^{\bar A}(\R^n))$ and $W^s_0 E^A(\Omega)$ as the closure of $C^\infty_c(\Omega)$ in norm topology.

Also, in \cite{FBSp24} it is shown that for any $A$, there exists a Banach space $X$ such that $X^*= W^s_0 L^A(\Omega)$ and that the extension of Proposition \ref{embeddings} to the fractional case holds:
\begin{prop}\label{embeddings-s}
Let $A$ be a Young function. Then the embedding $W^s_0 L^A(\Omega) \subset L^A(\Omega)$ is compact.
\end{prop}

\subsection{A review on the existence of minimizers for the energy function}

Given a Young function $A$ without any further assumption on its behavior,  in  \cite[Theorem 3.1]{GLMS99} (see also \cite[Theorem 4.1]{GM02}) it is proved that for any election of the normalization $\alpha>0$, the minimization problem $E(\alpha)$ is solvable:

\begin{prop} \label{mini.prop}
Let $E(\alpha)$ be the energy function defined in \eqref{minimi}. Then for every $\alpha>0$, there exists a function $u_\alpha\in W^1_0L^A(\Omega)$ such  that $E(\alpha)=\int_\Omega A(|\nabla u_\alpha|)\, dx$ and $\int_\Omega A(|u_\alpha|)\,dx=\alpha$.
\end{prop}

As mentioned in the introduction, if $A, \bar A\in \Delta_2$ it is straightforward to see that by means of the Lagrange's multiplier rule there exists $\lambda(\alpha)$ such that
\begin{equation} \label{form.debil}
\int_\Omega a(|\nabla u_\alpha|)\frac{\nabla u_\alpha \cdot \nabla v}{|\nabla u_\alpha|} \,dx = \lambda (\alpha) \int_\Omega a(|u_\alpha|)\frac{u_\alpha v}{|u_\alpha|}\,dx \qquad \text{ for all } v\in C_c^\infty(\Omega).
\end{equation}
That is, $u_\alpha$ solves the eigenvalue problem \eqref{autov} in the weak sense.

By using a generalized version of the Lagrange's multiplier rule, in \cite[Theorem 4.2]{GM02} the previous claim is proved to hold without any growth assumption on $A$:

\begin{prop} \label{autov.prop}
Let $\alpha>0$ and let $u_\alpha$ be a minimizer of $E(\alpha)$. Then there exists $\lambda(\alpha)>0$ for which $u_\alpha$  solves \eqref{form.debil}.
\end{prop}

As for the fractional problem, this was tackled in \cite{SV22} (see also \cite{FBSp24} where higher order eigenvalues were also studied) and the following result was otained.
\begin{prop}\label{mini.prop.s}
For any $\alpha>0$, there exists a minimizer $u_\alpha^s\in W^s_0L^A(\Omega)$ of the energy function $E^s(\alpha)$ defined in \eqref{minimi.nl.intro}. Moreover, there exists $\lambda^s(\alpha)$ such that $u_\alpha^s$ is a weak solution of \eqref{autov.nl.intro}, that is
$$
\iint_{\R^{2n}} a(|D^s u_\alpha^s|)\frac{D^s u_\alpha^s D^s v}{|D^s u_\alpha^s|}\, d\nu_n = \lambda^s(\alpha) \int_\Omega a(|u_\alpha^s|) \frac{u_\alpha^s v}{|u_\alpha^s|}\, dx, 
$$
for all $v\in C^\infty_c(\Omega)$.
\end{prop}

\section{Elementary estimates on the energy function}\label{sec.elementary}

In this section we will assume that the Young function $A$ verifies the $\Delta_2-$condition either at 0 or at $\infty$ and deduce some immediate bounds on the energy $E(\alpha)$ and on the eigenvalue $\lambda(\alpha)$.

To this end, let us consider $u\in C^\infty_c(\Omega)$ to be fixed such that
$$
\int_\Omega A(|u|)\, dx = 1.
$$
Then, since $A$ is continuous and nondecreasing, if we define the function
\begin{equation}\label{def.phi}
\phi(r) = \phi_A^u(r)=\int_\Omega A(r|u|)\, dx,
\end{equation}
it follows that $\phi\colon \R_+\to\R_+$ is continuous, nondecreasing, $\phi(0)=0$, $\phi(1)=1$ and $\phi(\infty)=\infty$. Hence, for any $\alpha>0$, there exists $r=r_\alpha>0$ such that $\phi(r_\alpha)=\alpha$.

Recall that if $\alpha\le 1$, then $r_\alpha\le 1$ and that if $\alpha\ge 1$ then $r_\alpha\ge 1$.

We want to estimate $r_\alpha$ in  terms of $\alpha$. To this end, we recall an easy lemma regarding Young functions
\begin{lem}\label{est.A}
Let $A$ be a Young function such that  $A\in \Delta_2^0$ (or $A\in \Delta_2^\infty$)   and let $p>1$ be defined in \eqref{def.p}.

Then it holds that $\tau^p A(t)\le A(\tau t)\le \tau A(t)$ for $0<\tau<1$ and $t<T_0$ (or $\tau A(t)\le A(\tau t)\le \tau^p A(t)$ for $\tau>1$ and $t>T_\infty$).

Hence, if $A\in \Delta_2$, then
$$
\min\{\tau, \tau^p\} A(t)\le A(\tau t)\le \max\{\tau, \tau^p\} A(t),\quad \text{for $\tau>0$}.
$$
\end{lem}
The proof of this lemma is elementary and is omitted.

With the help of Lemma \ref{est.A} we can obtain a simple estimate on $r_\alpha$, namely
\begin{lem}\label{est.r}
Let $A\in \Delta_2$ and $u\in C^\infty_c(\Omega)$. Let $\phi=\phi_A^u$ be defined in \eqref{def.phi}, $\alpha>0$ and let $r_\alpha>0$ be such $\phi(r_\alpha)=\alpha$. Then
$$
\min\{\alpha, \alpha^\frac{1}{p}\}\le r_\alpha\le \max\{\alpha, \alpha^{\frac{1}{p}}\},
$$
where $p$ is the constant defined in \eqref{def.p}.
\end{lem}

\begin{proof}
It is immediate from the definition of $r_\alpha$ and Lemma \ref{est.A}.
\end{proof}

Next, we combine Lemmas \ref{est.A} and \ref{est.r} to obtain a bound for the energy $E(\alpha)$.
\begin{lem}\label{est.E}
Given $\alpha>0$, the energy of level $\alpha$ is defined by \eqref{minimi}. Then, if $A\in \Delta_2$, we have that
$$
\min\{\alpha^p, \alpha^\frac{1}{p}\}E(1)\le E(\alpha)\le \max\{\alpha^p, \alpha^\frac{1}{p}\}E(1)
$$
\end{lem}

\begin{proof}
Immediate from the previous lemmas.
\end{proof}

Next, we want to compare the energy $E(\alpha)$ to the eigenvalue $\lambda(\alpha)$ of problem \eqref{autov}.
\begin{lem}\label{est.autov}
Let $A$ be a Young function such that $A\in \Delta_2$ and let $p$ be defined in \eqref{def.p}.

Let $\alpha>0$ be a normalization parameter and let $u\in W^1_0L^A(\Omega)$ be the minimizer associated to the energy $E(\alpha)$. Let $\lambda(\alpha)$ be the eigenvalue of \eqref{autov} associated to $u$. Then, we have
$$
\frac{1}{p}\frac{E(\alpha)}{\alpha}\le \lambda(\alpha)\le p\frac{E(\alpha)}{\alpha}.
$$
\end{lem}

\begin{proof}
First, observe that from \cite[Theorem 4.23]{GM02}, there exists a minimizer $u\in W^1_0L^A(\Omega)$ for $E(\alpha)$ and $u$ is a weak solution to \eqref{autov}.

Now, if we use $u$ as a test function in the weak formulation of \eqref{autov}, we get that
\begin{equation}\label{esti1}
\lambda(\alpha) = \frac{\int_\Omega a(|\nabla u|)|\nabla u|\, dx}{\int_\Omega a(|u|)|u|\, dx}.
\end{equation}

Next, just observe that we have the inequality
\begin{equation}\label{esti2}
A(t)\le a(t) t\le p A(t),
\end{equation}
where the first inequality follows from the monotonicity of $a(t)$ and the second one is the $\Delta_2-$condition.

The result now follows just combining \eqref{esti1} and \eqref{esti2}.
\end{proof}

\begin{rem}
Observe that combining Lemmas \ref{est.E} and \ref{est.autov} we immediately get that there exist $c, C>0$ such that
$$
c\min\{\alpha^{p-1}, \alpha^{\frac{1}{p}-1}\} \le \frac{E(\alpha)}{\alpha}, \lambda(\alpha)\le C\max\{\alpha^{p-1}, \alpha^{\frac{1}{p}-1}\}. 
$$
Also observe that these bounds trivialize when $\alpha\to 0$ and when $\alpha\to\infty$.
\end{rem}

\section{Continuity and differentiability of the energy function}\label{sec.regularity}

In this section we study the regularity properties of the energy function $E(\alpha)$ defined in \eqref{minimi}. First, without assuming the $\Delta_2$ condition on $A$ nor $\bar A$, we prove the continuity of $E(\alpha)$

\begin{thm}\label{teo.cont}
The energy function $E(\alpha)$ is continuous for every $\alpha>0$.
\end{thm}

We split the proof into two lemmas.

\begin{lem} \label{lema.usc}
Let $A$ be a Young function and let $E(\alpha)$ be the energy function defined in \eqref{minimi}. Then, $E$ is upper semicontinuous.
\end{lem}

\begin{proof}
Let $\alpha_0>0$ be fixed. Given $\ve>0$, take $u_0=u_{\ve,\alpha_0}\in C^\infty_c(\Omega)$ be such that
$$
\int_\Omega A(|\nabla u_0|)\, dx \le E(\alpha_0) + \ve\quad\text{and}\quad \int_\Omega A(|u_0|)\, dx = \alpha_0.
$$
Let $\phi(r)=\phi_A^{u_0}(r)$ be the function defined in \eqref{def.phi} and given $\alpha>0$, let $r_\alpha>0$ be such that $\phi(r_\alpha)=\alpha$.

Observe that  $\lim_{\alpha\to \alpha_0} \phi(r_\alpha) = \lim_{\alpha\to \alpha_0} \alpha = \alpha_0 = \phi(1)$. But, since $\phi$ is strictly increasing and continuous, it follows that $\phi^{-1}$ is also continuous and strictly increasing, from where it follows that $r_\alpha\to 1$ when $\alpha\to \alpha_0$. 

Hence $r_\alpha u_0$ is an admissible test for $E(\alpha)$ and we get that
$$
\limsup_{\alpha\to \alpha_0} E(\alpha) \leq \limsup_{\alpha\to \alpha_0} \int_\Omega A (r_\alpha |\nabla u_0|)\, dx = \int_\Omega A(|\nabla u_0|)\,dx \le E(\alpha_0) + \ve.
$$
Since $\ve>0$ is arbitrary, the result follows.
\end{proof}

\begin{lem} \label{lema.lsc}
Let $A$ be a Young function and let $E(\alpha)$ be the energy function defined in \eqref{minimi}. Then, $E$ is lower semicontinuous.
\end{lem}

\begin{proof}
Let $\alpha_0>0$ be fixed and consider $\alpha_k\to\alpha_0$ as $k\to\infty$. Let $\ve>0$ and for each $k\in\N$ let $u_k\in C^\infty_c(\Omega)$ be such that
$$
\int_\Omega A(|\nabla u_k|)\,dx \le E(\alpha_k) + \ve \quad \text{and} \quad \int_\Omega  A(|u_k|)\,dx = \alpha_k.
$$
Observe that by Lemma \ref{lema.usc}, $E(\alpha_k)$ is bounded, and hence $\int_\Omega A(|\nabla u_k|)\, dx$ is bounded. Therefore, by Lemma \ref{cota.prelim}, $\{u_k\}_{k\in\N}$ is bounded in $W^1_0L^A(\Omega)$.

This allows to apply Proposition \ref{embeddings} to derive the existence of a function  $u\in W^1_0L^A(\Omega)$ such that
$$
u_k \rightharpoonup u \text{ weakly* in } W^1_0L^A(\Omega) \quad \text{and}\quad u_k \to u \text{ strongly in } L^A(\Omega).
$$
From this we  conclude that $\int_\Omega A(|u|)\,dx = \alpha_0$ and hence, it is an admissible competitor for $E(\alpha_0)$, from where
$$
E(\alpha_0) \leq \int_\Omega A(|\nabla u|)\,dx \leq \liminf_{k\to\infty}  \int_\Omega A(|\nabla u_k|)\,dx \le \liminf_{k\to\infty}  E(\alpha_k) + \ve,
$$
where for the second inequality we have used \cite[Lemma 2.2]{GLMS99}. This estimate concludes the proof.
\end{proof}

From the proof of Theorem \ref{teo.cont} one immediately obtains the following information:
\begin{prop} \label{propo.lim}
Let $\alpha_0>0$ and let $\alpha_k\to\alpha_0$ as $k\to\infty$. Let $\{u_k\}_{k\in \N}$ be a sequence of minimizers of $E(\alpha_k)$. Then, $\{u_k\}_{k\in\N}$ is weak* precompact in $W^1_0L^A(\Omega)$ and each accumulation point $u\in W^1_0L^A(\Omega)$ of the sequence in the weak*-topology is a minimizer for $E(\alpha_0)$.

Moreover, if $u_{k_j}\rightharpoonup u $ weakly* in $W^1_0 L^A(\Omega)$, then
$$
\int_\Omega A(|\nabla u|)\, dx = \lim_{j\to\infty} \int_\Omega A(|\nabla u_{k_j}|)\, dx.
$$
\end{prop}

Now, we analyze the differentiability of $E(\alpha)$. We stress that no $\Delta_2-$condition is imposed on $A$ nor $\bar A$. In fact, we prove the following:
\begin{thm}\label{E.Lip}
Let $A$ be a Young function and let $E(\alpha)$ be the energy function defined in \eqref{minimi}. Then $E(\alpha)$ is Lipschitz continuous. Moreover, 
$$
0\le \liminf_{\alpha\to\alpha_0} \frac{E(\alpha)-E(\alpha_0)}{\alpha-\alpha_0} \le \limsup_{\alpha\to\alpha_0} \frac{E(\alpha)-E(\alpha_0)}{\alpha-\alpha_0}\le \lambda(\alpha).
$$
\end{thm}

\begin{proof}
Observe that the lower bound is immediate, since $E(\alpha)$ is increasing.

Now, Let $\alpha_0>0$ and let $u_0\in W^1_0L^A(\Omega)$ be a minimizer of $E(\alpha_0)$, that is, 
$$
E(\alpha_0)=\int_\Omega A(|\nabla u_0|)\,dx \quad \text{and} \quad \int_\Omega A(|u_0|)\,dx = \alpha_0.
$$
Define $F(r,\alpha) =\int_\Omega A(r |u_0|)\,dx - \alpha$. Observe that
$$
\frac{\partial F}{\partial r}(r, \alpha)= \int_\Omega a(r|u_0|)|u_0|\,dx, \qquad 
\frac{\partial F}{\partial \alpha}(r, \alpha)= -1.
$$
Moreover, since $u_0\neq 0$, $\frac{\partial F}{\partial r}(1,\alpha_0)= \int_\Omega a(|u_0|)|u_0|\,dx>0$ . Then we can apply the Implicit Function Theorem to obtain the existence of a function $r(\alpha)$ defined for $\alpha \in I:=(\alpha_0-\delta, \alpha_0+\delta)$ for some $\delta>0$, for which $r(\alpha_0)=1$, $F(r(\alpha),\alpha)=0$ for any $\alpha\in I$. Moreover,
$$
r'(\alpha_0) = -\frac{\frac{\partial F}{\partial \alpha}(1,\alpha_0)}{\frac{\partial F}{\partial r}(1,\alpha_0)} = \frac{1}{\int_\Omega a(|u_0|)||u_0|\,dx}>0.
$$
Observe that the function $G(\alpha):= \int_\Omega A(r(\alpha)|\nabla u_0|)\, dx$, $\alpha\in I$, near $\alpha_0$ behaves as
\begin{align*}
G(\alpha)&=G(\alpha_0)+ G'(\alpha_0)(\alpha-\alpha_0)+ o(\alpha-\alpha_0)\\
&=E(\alpha_0) + (\alpha-\alpha_0)r'(\alpha_0) \int_\Omega a(|\nabla u_0|)|\nabla u_0| \, dx  + o(\alpha-\alpha_0).
\end{align*}
Since $F(r(\alpha),\alpha)=0$ for any $\alpha\in I$, the function $r(\alpha) u_0$ is admissible in the minimization problem $E(\alpha)$ for any $\alpha\in I$, giving that
$$
\frac{E(\alpha)-E(\alpha_0)}{\alpha-\alpha_0}\leq \frac{G(\alpha)-E(\alpha_0)}{\alpha-\alpha_0} = \frac{\int_\Omega a(|\nabla u_0|)|\nabla u_0|\,dx }{\int_\Omega a(|u_0|)|u_0|\,dx} + o(1) = \lambda(\alpha_0) + o(1),
$$ 
as we wanted to prove.
\end{proof}

In order to obtain the differentiability of the energy function, we need to impose some technical, but somewhat natural additional conditions on $A$.
\begin{thm} \label{teo.dif}
Let $A$ be a Young function such that $t\mapsto a(t)t$ is convex.

Let $E(\alpha)$ be the energy function defined in \eqref{minimi}. Then $E(\alpha)$ is differentiable for any $\alpha>0$. Moreover, 
\begin{align*}
E'(\alpha)=\lambda(\alpha), \qquad \left(\frac{E(\alpha)}{\alpha}\right)'= \frac{1}{\alpha} \left(\lambda(\alpha)-\frac{E(\alpha)}{\alpha}\right),
\end{align*}
where $\lambda(\alpha)$ is the eigenvalue defined in \eqref{autov}.
\end{thm}

\begin{proof}
In view of Theorem \ref{E.Lip}, it remains to see that
$$
\liminf_{\alpha\to\alpha_0} \frac{E(\alpha)-E(\alpha_0)}{\alpha-\alpha_0}\ge\lambda(\alpha_0).
$$

Let $\{\alpha_k\}_{k\in\N}$ be such that $\alpha_k\to \alpha_0$ when $k\to\infty$, and
$$
\liminf_{\alpha\to \alpha_0} \frac{E(\alpha)-E(\alpha_0)}{\alpha-\alpha_0}=\lim_{k\to \infty} \frac{E(\alpha_k)-E(\alpha_0)}{\alpha_k-\alpha_0}.
$$
Arguiying analogously as in the proof of Theorem \ref{teo.cont}, we can assume that there exists $\{u_k\}_{k\in\N}\subset W^1_0L^A(\Omega)$ such that
$$
E(\alpha_k) = \int_\Omega A(|\nabla u_k|)\, dx,\qquad \int_\Omega A(|u_k|)\, dx = \alpha_k
$$
and that
\begin{equation} \label{converg}
\begin{split}
&u_k \rightharpoonup u \text{ weakly* in } W^1_0L^A(\Omega) \quad \text{ and } \\
&u_k \to u \text{ strongly in } L^A(\Omega) \text{ and pointwisely}.
\end{split}
\end{equation}
Moreover, by  Proposition \ref{propo.lim}, $u\in W^1_0L^A(\Omega)$ is a minimizer for $E(\alpha_0)$.

Now, proceeding as Theorem \ref{E.Lip} for each $k\in\N$ there exists a function $r_k(\alpha)$ defined in $t\in I_k :=(\alpha_k-\delta, \alpha_k+\delta)$ for some $\delta>0$ for which
$$
r_k(\alpha_k)=1, \quad \int_\Omega  A(r_k(\alpha) |u_k|)\,dx=\alpha, \quad r_k'(\alpha_k) = \frac{1}{\int_\Omega a(|u_k|)|u_k|\, dx} > 0.
$$
Hence, testing in $E(\alpha_0)$ with $r_k(\alpha_0) u_k$ and performing a Taylor expansion around $\alpha=\alpha_k$ we get that
\begin{align*}
E(\alpha_0)&\leq \int_\Omega A(r_k(\alpha_0) |\nabla u_k|)\,dx\\ 
&= E(\alpha_k) + (\alpha_0-\alpha_k) r_k'(\alpha_k) \int_\Omega a(|\nabla u_k|)|\nabla u_k|\,dx + o(\alpha_0-\alpha_k).
\end{align*}
Observe that, since $a(t)$ is continuous and $u_k$ is converging pointwisely, it follows that $o(\alpha_0-\alpha_k)$ is independent of $k\in\N$.

Therefore,
\begin{align*}
\lim_{k\to\infty} \frac{E(\alpha_k) - E(\alpha_0)}{\alpha_k-\alpha_0} &\geq \liminf_{k\to\infty}\frac{\int_\Omega a(|\nabla u_k|)|\nabla u_k|\,dx}{\int_\Omega a(|u_k|)|u_k|\,dx}\\
&\ge \frac{\int_\Omega a(|\nabla u|)|\nabla u|\,dx}{\int_\Omega a(|u|)|u|\,dx} = \lambda(\alpha_0),
\end{align*}
where we have used  \cite[Lemma 2.2]{GLMS99} and our convexity assumption on the function $t\mapsto a(t)t$.

Finally, using the expression for $E'(\alpha)$ we get that
\begin{align*}
\left(\frac{E(\alpha)}{\alpha}\right)'&=\frac{1}{\alpha}\left(E'(\alpha)-\frac{E(\alpha)}{\alpha}\right)= \frac{1}{\alpha} \left(\lambda(\alpha)-\frac{E(\alpha)}{\alpha}\right).
\end{align*}
The proof is complete.
\end{proof}

\begin{rem}
Observe that the eigenvalue $\lambda(\alpha)$ and the energy quotient $E(\alpha)/\alpha$ coincide if and only if they are constants. This is exactly the case where $A$ is homogeneous, that is $A(t)=t^p$ for some $p>1$.
\end{rem}

\section{Sharp asymptotic behavior when $\alpha\to 0$ and $\alpha\to\infty$}\label{sec.asymptotic}

In order to analyze the precise behavior at $0$ and at $\infty$, we need to assume that $A\in \Delta_2^i$ at $i=0$ and $i=\infty$ respectively. Moreover, we need to impose the slightly stronger assumption that the limit defining $M_i(t)$ in \eqref{matu} exists.

So, fixed a normalization parameter $\tau>0$ we define the minimization problem for the associated Matuszewska functions,
\begin{align}  \label{ee}
E_i(\tau) &= \inf\left\{\int_\Omega M_i(|\nabla u|)\, dx\colon u\in C_c^\infty(\Omega), \, \int_\Omega M_i(|u|)\, dx = \tau\right\}.
\end{align}

We then obtain the following result.
\begin{thm} \label{teo1}
Let $A$ be a Young function satisfying the $\Delta_2^i$ condition for $i=0$ or $i=\infty$ and assume that the corresponding limit defining $M_i(t)$ in \eqref{matu} exists. Then, it holds that:
$$
\limsup_{\alpha\to i} \frac{E(\alpha)}{\alpha} = \inf_{\tau>0} \frac{E_i(\tau)}{\tau}.
$$
\end{thm}

\begin{proof}
We work in the case where $i=\infty$, the other case being identical. We will split the proof into two steps.

\emph{Step 1: Upper bound}.

Let $u\in C^\infty_c(\Omega)$ be a fixed function such that $u\neq 0$ and consider the function $\phi = \phi_A^u$ defined in \eqref{def.phi}.  Hence, given $\alpha>0$ there exists a unique $r_\alpha>0$ such that $\phi(r_\alpha)=\alpha$. Since $\phi$ is continuous and $u$ is fixed we have that $r_\alpha\to \infty$ if $\alpha\to \infty$.

Since we can write
$$
\frac{\alpha}{A(r_\alpha)} = \int_\Omega \frac{A(r_\alpha|u|)}{A(r_\alpha)}\, dx,
$$
according to the definition \eqref{matu} and the Dominated Convergence Theorem, we get that
\begin{equation} \label{limite1}
\lim_{\alpha\to\infty} \frac{\alpha}{A(r_\alpha)} = \int_\Omega M_\infty(|u|)\, dx.
\end{equation}

The equality $\phi(r_\alpha)=\alpha$ implies that the function $r_\alpha u$ is admissible in the minimization problem for $E(\alpha)$, therefore
\begin{equation}\label{cota.infty}
\begin{split}
\limsup_{\alpha\to \infty} \frac{E(\alpha)}{\alpha} &\le \lim_{\alpha\to\infty} \frac{1}{\alpha}\int_\Omega A(r_\alpha|\nabla u|)\, dx\\
&= \lim_{\alpha\to\infty} \frac{A(r_\alpha)}{\alpha}\int_\Omega \frac{A(r_\alpha|\nabla u|)}{A(r_\alpha)}\, dx\\
&=\frac{\int_\Omega M_\infty(|\nabla u|)\, dx}{\int_\Omega M_\infty(|u|)\, dx},
\end{split}
\end{equation}
where we have used \eqref{limite1} and  \eqref{matu}

Since $u\in C_c^\infty(\Omega)$ is arbitrary, from \eqref{cota.infty} it is deduced that
$$
\limsup_{\alpha\to\infty} \frac{E(\alpha)}{\alpha} \le \inf_{\tau>0} \frac{E_\infty(\tau)}{\tau}
$$
as needed.

\medskip

\emph{Step 2: Lower bound}.
Let us derive the lower bound for the ratio $E(\alpha)/\alpha$. For this, let $u\in C^\infty_c(\Omega)$ be fixed, and as before, let $r_\alpha>0$ be such that $\phi(r_\alpha)=\alpha$.

For $\ve>0$ we have that
$$
\inf_{\tau>0} \frac{E_\infty(\tau)}{\tau} - \ve \le  \frac{\int_\Omega M_\infty(|\nabla u|)\, dx}{\int_\Omega M_\infty(|u|)\, dx}-\ve.
$$
Then, according to \eqref{matu}, there exists $\alpha_\infty = \alpha_\infty(\ve, u)$ such that for $\alpha>\alpha_\infty$,
$$
\frac{\int_\Omega M_\infty(|\nabla u|)\, dx}{\int_\Omega M_\infty(|u|)\, dx}-\ve \le \frac{\int_\Omega \frac{A(r_\alpha|\nabla u|)}{A(r_\alpha)}\, dx}{\int_\Omega \frac{A(r_\alpha|u|)}{A(r_\alpha)}\, dx} = \frac{1}{\alpha}\int_\Omega A(r_\alpha|\nabla u|)\, dx.
$$
Taking infimum over $u\in C^\infty_c(\Omega)$, we conclude that
$$
\inf_{\tau>0} \frac{E_\infty(\tau)}{\tau} - \ve \le \frac{E(\alpha)}{\alpha}\quad \text{for } \alpha>\alpha_\infty,
$$
from where we deduce that
$$
\inf_{\tau>0} \frac{E_\infty(\tau)}{\tau} \le \liminf_{\alpha\to \infty} \frac{E(\alpha)}{\alpha}
$$
and this concludes the proof.
\end{proof}

Lemma \ref{lema.usc}, and Theorem \ref{teo1} lead to the uniform boundedness of the ratio $E(\alpha)/\alpha$:
\begin{cor}
It holds that
$$
\sup_{\alpha>0} \frac{E(\alpha)}{\alpha} <\infty.
$$
\end{cor}

As a direct consequence, we have the following.
\begin{cor}
With the same assumptions of Theorem \ref{teo1}, let $p_i$ be the exponent given by Proposition \ref{M.potencia}. Then, for $i=0$ or $i=\infty$,
$$
\lim_{\alpha\to i} \frac{E(\alpha)}{\alpha} = \lambda_{p_i},
$$
where $\lambda_p$ denotes the first eigenvalue of the $p-$Laplacian with Dirichlet boundary conditions in $\Omega$.
\end{cor}

\section{The behavior of $E(\alpha)/\alpha$ without the $\Delta_2-$condition}\label{sec.not.delta2}
In this section we explore the limit behavior of the energy quotient $E(\alpha)/\alpha$ in the case where the Young function $A$ does not satisfy the $\Delta_2^i$ condition for $i=0$ or $i=\infty$.

We will assume the stronger condition \eqref{s.not.delta2}, and hence, by Proposition \ref{Minf.trivial}, the Matuszewska function $M_i(t)$ is trivial. That is
$$
\lim_{s\to i} \frac{A(ts)}{A(s)} = M_i(t) = \begin{cases}
0 & \text{if } 0\le t<1\\
1 & \text{if } t=1\\
\infty & \text{if } t>1.
\end{cases}
$$

In this case, we have the following result.
\begin{thm}\label{lim.sin.delta2}
Let $A$ be a Young function satisfying \eqref{s.not.delta2} for $i=0$ or $i=\infty$. Then, if $\Omega\subset \R^n$ has inner radius larger that 1,
$$
\lim_{\alpha\to i} \frac{E(\alpha)}{\alpha} = 0.
$$
\end{thm}

\begin{proof}
Recall that the inner radius of $\Omega$ is defined as
$$
R_\Omega = \sup\{r>0\colon B_r(x)\subset\Omega \text{ for some } x\in\Omega\}.
$$
If $R_\Omega>1$, then there exists $r>0$, $x_0\in\Omega$ with $B_r(x_0)=B_r\subset \Omega$ and $u\in C^\infty_c(\Omega)$ such that
$$
u(x)=1 \text{ for }x\in B_r,\quad 0\le u(x)\le 1 \quad \text{and}\quad |\nabla u(x)|< 1 \text{ for } x\in \Omega.
$$
Then, as usual, for $\alpha>0$ we define $r_\alpha$ such that $\phi(r_\alpha)=\phi_A^u(r_\alpha)=\alpha$ and therefore
$$
\frac{E(\alpha)}{\alpha}\le \frac{A(r_\alpha)}{\alpha}\int_\Omega \frac{A(r_\alpha |\nabla u|)}{A(r_\alpha)}\,dx = \frac{\int_\Omega \frac{A(r_\alpha |\nabla u|)}{A(r_\alpha)}\,dx}{\int_\Omega \frac{A(r_\alpha |u|)}{A(r_\alpha)}\,dx}.
$$
But, since $r_\alpha\to i$ as $\alpha\to i$, it follows from Proposition \ref{Minf.trivial} that
$$
\int_\Omega \frac{A(r_\alpha |u|)}{A(r_\alpha)}\,dx\to |\{u=1\}| = |B_r|>0\quad\text{and}\quad \int_\Omega \frac{A(r_\alpha |\nabla u|)}{A(r_\alpha)}\,dx\to 0,
$$
and the proof is complete.

\end{proof}

\section{Nonlocal case}\label{sec.nonlocal}

In this section we analyze the problem of the behavior of the energy function for the nonlocal case $E^s(\alpha)$ defined in \eqref{minimi.nl.intro}.

All the program developed in Sections \ref{sec.elementary} to \ref{sec.asymptotic} can be carried out without any change so we just include the results leaving the details to the interested reader. 

The results obtained in this case can be summarized in the following theorem

\begin{thm}\label{nonlocal.case}
Let $A$ be a Young function. Then the nonlocal energy function $E^s(\alpha)$ defined in \eqref{minimi.nl.intro} is Lipschitz continuous. If in addition the function $t\mapsto ta(t)$ is convex, then $E^s(\alpha)$ is differentiable and
$$
\frac{d}{d\alpha} E^s(\alpha) = \lambda^s(\alpha),\qquad \frac{d}{d\alpha} \left(\frac{E^s(\alpha)}{\alpha}\right) = \frac{1}{\alpha}\left(\lambda^s(\alpha) - \frac{E^s(\alpha)}{\alpha}\right),
$$
where $\lambda^s(\alpha)$ is the eigenvalue defined in \eqref{autov.nl.intro}.

Moreover, when $A\in \Delta_s^i$ for $i=0$ or $i=\infty$, and the limit in \eqref{matu} exists, then
$$
\lim_{\alpha\to i} \frac{E^s(\alpha)}{\alpha} = \inf_{\tau>0} \frac{E^s_i(\tau)}{\tau}
$$
where
$$
E^s_i(\tau) = \inf\left\{\iint_{\R^{2n}} M_i(|D^s u|)\, d\nu_n \colon u\in C^\infty_c(\Omega),\ \int_\Omega M_i(|u|)\, dx = \tau\right\}.
$$
\end{thm}

\section*{Acknowledgements}
This work was partially supported by ANPCyT under grant PICT 2019-3837 and by CONICET under grant  PIP 11220150100032CO. Both authors are members of CONICET and are grateful for the support.

\bibliographystyle{amsplain}
\bibliography{Biblio}

\end{document}